\documentclass[a4paper,reqno]{amsart}
\usepackage{amsmath,amssymb,stmaryrd}
\usepackage{hyperref}
\usepackage[capitalize]{cleveref}
\usepackage[all]{xy}

\theoremstyle{plain}
\numberwithin{equation}{section}
\newtheorem{thm}[equation]{Theorem}
\newtheorem{lem}[equation]{Lemma}
\newtheorem{cor}[equation]{Corollary}
\newtheorem{prop}[equation]{Proposition}
\theoremstyle{definition}
\newtheorem{df}[equation]{Definition}
\theoremstyle{remark}
\newtheorem{rem}[equation]{Remark}
\crefname{thm}{Theorem}{Theorems}
\crefname{lem}{Lemma}{Lemmas}
\crefname{cor}{Corollary}{Corollaries}
\crefname{prop}{Proposition}{Propositions}
\crefname{df}{Definition}{Definitions}
\crefname{rem}{Remark}{Remarks}
\crefname{eg}{Example}{Examples}
\crefname{equation}{}{}

\newcommand{\bC}{\mathbb{C}}
\newcommand{\rD}{\mathrm{D}}
\newcommand{\bR}{\mathbb{R}}
\newcommand{\bZ}{\mathbb{Z}}
\newcommand{\fT}{\mathfrak{T}}
\newcommand{\cA}{\mathcal{A}}
\newcommand{\cB}{\mathcal{B}}
\newcommand{\cC}{\mathcal{C}}
\newcommand{\cD}{\mathcal{D}}
\newcommand{\cE}{\mathcal{E}}
\newcommand{\cF}{\mathcal{F}}
\newcommand{\cH}{\mathcal{H}}
\newcommand{\fH}{\mathfrak{H}}
\newcommand{\cK}{\mathcal{K}}
\newcommand{\cM}{\mathcal{M}}
\newcommand{\mrm}{\mathrm{m}}
\newcommand{\cO}{\mathcal{O}}
\newcommand{\cS}{\mathcal{S}}
\newcommand{\cU}{\mathcal{U}}
\newcommand{\cW}{\mathcal{W}}
\newcommand{\fW}{\mathfrak{W}}
\newcommand{\cX}{\mathcal{X}}
\newcommand{\cZ}{\mathcal{Z}}
\providecommand{\1}{\mathrm{1}\hspace{-0.25em}\mathrm{l}} 
\renewcommand{\1}{\mathrm{1}\hspace{-0.25em}\mathrm{l}} 
\providecommand{\barotimes}{\mathbin{\bar{\otimes}}}
\newcommand{\wh}{\widehat}
\newcommand{\wt}{\widetilde}

\DeclareMathOperator{\id}{id}
\DeclareMathOperator{\ev}{ev}
\DeclareMathOperator{\pr}{pr}
\DeclareMathOperator{\Ker}{Ker}
\DeclareMathOperator{\Aut}{Aut}
\DeclareMathOperator{\Rep}{Rep}
\DeclareMathOperator{\Irr}{Irr}
\DeclareMathOperator{\Ad}{Ad} 

\begin{document}
\title{Discrete quantum subgroups of complex semisimple quantum groups}
\author{Kan Kitamura}
\date{}
\address{Graduate School of Mathematical Sciences, The University of Tokyo, 3-8-1 Komaba Meguro-ku Tokyo 153-8914, Japan}
\email{kankita@ms.u-tokyo.ac.jp}
\begin{abstract}
	We classify discrete quantum subgroups in the quantum double of the $q$-deformation of a compact semisimple Lie group, regarded as the complexification. 
	We also record their classifications in some variants of quantum groups.
	Along the way, we show that quantum doubles of non-Kac type compact quantum groups do not admit the quantum analog of lattices considered by Brannan--Chirvasitu--Viselter. 
\end{abstract}
\maketitle

	\section{Introduction}\label{sec:intro}
	
	The Drinfeld--Jimbo $q$-deformation of the universal enveloping algebra of a semisimple Lie algebra gives the first example of quantum groups. 
	The corresponding deformation of the C*-algebra of continuous functions on the compact Lie group is realized as a compact quantum group by Woronowicz. 
	For the $q$-deformation $K_q$ of a compact semisimple Lie group $K$, it has been observed that the quantum double $G_q=\rD(K_q)$ of $K_q$ can be regarded as the $q$-deformation of the complexification $G$ of $K$ as in \cite{Podles-Woronowicz1990:quantum,Pusz1993:irrepqLorentz,Arano:sphericalDrinfeld,Voigt-Yuncken2020:book}. 
	
	Analytic properties of a locally compact group, most notably Kazhdan's property (T), are inherited by its lattices. 
	In the quantum situation, the quantum double $\rD(K)$ of a compact quantum group $K$ contains its Pontryagin dual $\wh{K}$ as a discrete quantum subgroup. 
	For $\wh{K}$, central versions of several approximation properties can be considered as investigated in \cite{Freslon2013:examples,deCommer-Freslon-Yamashita2014:CCAP,Arano:sphericalDrinfeld}, 
	which are also closely related to the original approximation properties of $\rD(K)$. 
	On the other hand, $G_q$ in our picture decomposes into $K_q$ and $\wh{K}_q$, corresponding to $K$ and $AN$, respectively, from the Iwasawa decomposition $G=KAN$ in the classical setting. 
	Here $\wh{K}_q$ is discrete unlike $AN$. In this sense, $G_q$ has a discrete quantum subgroup with a different appearance from arithmetic subgroups, which essentially exhaust all lattices in most of the classical cases by Margulis' arithmeticity theorem. 
	
	With motivation from the quantum analog of lattices, we classify discrete closed quantum subgroups in $G_q$ for a connected simply connected compact semisimple Lie group $K$ and $0<q<1$. 
	In particular, it follows that when $G$ is simple, the discrete quantum subgroups of $G_q$ that are not finite subgroups are always ``commensurable" with $\wh{K}_q$. In other words, they share common finite index quantum subgroups with $\wh{K}_q$. 
	See \cref{thm:classifyss} for the precise form of the classification. 
	There is another approach by Christian Voigt to discrete quantum subgroups of $G_q$. 
	In \cite{Hoshino2023:equivariant}, the classification of finite index quantum subgroups of $\wh{K}_q$ is contained. 
	
	We should also mention the work by Brannan--Chirvasitu--Viselter~\cite{Brannan-Chirvasitu-Viselter2020:actionslattices}, where a quantum analog of lattices in locally compact quantum groups was considered. 
	They investigated several permanence properties and showed that $\rD(K)$ for a Kac type compact quantum group $K$ has $\wh{K}$ as its lattice. 
	However, when $K$ is not of Kac type, we show that $\rD(K)$ does not admit any lattices in this sense. 
	Especially, $G_q$ for $0<q<1$ does not admit them in contrast to the classical situation. 
	
	Our strategy is based on the simple observation that any discrete quantum subgroup of $G_q$ must be contained in the semidirect product of the maximal torus $T$ of $K$ and $\wh{K}_q$, 
	which relies on the classification of irreducible $*$-representations of $C(K_q)$ due to Soibelman~\cite{Soibelman1990:algebraquantum,Korogodski-Soibelman1998:bookptI}. 
	Here, we note that the semidirect product of $T$ and $\wh{K}_q$ gives the quantum Borel subgroup of $G_q$ in view of the Iwasawa decomposition. 
	This reduction makes the situation much easier than lattices in the classical case. 
	Then we conduct the classification by rewriting the involved data in terms of the fusion ring of a compact group. 
	We can also apply our method for variants of $G_q$, such as the quantum double $\rD(U_q(n))$ of the $q$-deformation of a unitary group, to classify their discrete quantum subgroups. 
	The observation still works in a more general setting of a maximal Kac quantum subgroup defined by \cite{Soltan:quantumBohr} after an idea due to Stefaan Vaes, 
	to show the absence of the analog of lattices by \cite{Brannan-Chirvasitu-Viselter2020:actionslattices}. 
	To treat the non-coamenable cases, we use the techniques of the Fourier algebras of quantum groups as utilized in \cite{Daws-Kasprzak-Skalski-Soltan:closed}. 
	
	The plan of this paper is as follows. 
	In \cref{sec:preliminary}, we recall basic concepts about locally compact quantum groups and fix some notation and terminologies. 
	We begin \cref{sec:nolattice} with some technical aspects of quantum doubles to settle the meaning of discrete quantum subgroups in quantum doubles (\cref{prop:DQGclVW}). 
	Then we show the non-existence of lattices in the sense of \cite{Brannan-Chirvasitu-Viselter2020:actionslattices} (\cref{thm:nolattice}). 
	The reader who only wants to see the classification of discrete quantum subgroups of $G_q$ can skip this section except \cref{rem:subDQGtorus}. 
	In \cref{sec:classify}, we give the main classification result of discrete quantum subgroups in $G_q$ (\cref{thm:classifyss}). 
	\cref{sec:variant} treats the cases of several quantum groups akin to $G_q$, such as variants of the quantum Lorentz group appeared in \cite{Pusz1993:irrepqLorentz} (\cref{cor:classifyLorentz}), and $\rD(U_q(n))$ (\cref{thm:classifyGLqC}). 
	
	\subsection*{Acknowledgments}
	The author would like to express his deepest appreciation to Christian Voigt for discussions on various topics of quantum groups, including quantum analog of lattices and homogeneous spaces, and his comments on the early manuscript. 
	He would like to thank Jacek Krajczok for discussions on quantum doubles and Kac type quantum groups. He would like to thank his advisor Yasuyuki Kawahigashi for his support and encouragement. 
	Also, he would like to thank the University of Glasgow, where a large part of this manuscript was written during the visit, for all the hospitality. 
	This work was supported by JSPS KAKENHI Grant Number JP21J21283, JST CREST program JPMJCR18T6, JSPS Overseas Challenge Program for Young Researchers, and the WINGS-FMSP program at the University of Tokyo. 
	
	\section{Preliminaries}\label{sec:preliminary}
	
	Every tensor product of C*-algebras in this paper is taken with respect to the minimal C*-norm. 
	We use the symbol $\barotimes$ for a tensor product of von Neumann algebras and $\odot$ for an algebraic tensor product when we want to emphasize it. 
	We often abuse notation to still write $f$ for the extension of a bounded linear map $f\colon A\to \cM(B)$ to $\cM(A)$ that is strictly continuous on the unit ball if it exists. 
	Also, we use the so-called leg-numbering notation. 
	The symbol $\cZ$ shall indicate the center of an algebra or a group. 
	
	We recall locally compact quantum groups by Kustermans--Vaes~\cite{Kustermans-Vaes:lcqg,Kustermans-Vaes:lcqgvN,Vaes:thesis}. 
	\begin{df}\label{def:LCQG}
		A \emph{locally compact quantum group} is a pair $G=(L^\infty(G),\Delta_G)$ of a von Neumann algebra $L^\infty(G)$ and a unital normal $*$-homomorphism $\Delta_G\colon L^\infty(G)\to L^\infty(G)\barotimes L^\infty(G)$, called the comultiplication, with the following conditions. 
		\begin{itemize}
			\item
			$\Delta_G$ is coassociative in the sense of $(\Delta_G\otimes\id)\Delta_G=(\id\otimes\Delta_G)\Delta_G$. 
			\item
			There are faithful normal semifinite weights $h_l,h_r$ on $L^\infty(G)$ 
			such that 
			$(\omega\otimes h_l)\Delta_G=h_l$ and $(h_r\otimes\omega)\Delta_G=h_r$ on $L^\infty(G)_+$ for any normal state $\omega$ on $L^\infty(G)$. 
		\end{itemize}
	\end{df}
	
	The GNS construction of $h_l$ gives a faithful normal $*$-representation of $L^\infty(G)$ on a Hilbert space denoted by $L^2(G)$. 
	We have a well-defined unitary $W^G\in \cU(L^2(G)\otimes L^2(G))$ satisfying $W^{G*}\Lambda(x)\otimes\Lambda(y)= (\Lambda\otimes\Lambda) \bigl(\Delta_G(y)(x\otimes1)\bigr)$, 
	where $\Lambda\colon \{ x\in L^\infty(G) \,|\, h_l(x^*x)<\infty \}\to L^2(G)$ is the map from the GNS construction. 
	This is a multiplicative unitary in the sense of $W^G_{23}W^G_{12}=W^{G}_{12}W^G_{13}W^G_{23}$. 
	It satisfies $\Ad W^{G*}(1\otimes x )=\Delta_G(x)$ for $x\in L^\infty(G)$. 
	
	We mainly use the C*-algebraic picture $C_0(G)$ of $G$, given by the norm closure of $\{ (\id\otimes\omega)(W^G) \,|\, \omega\in\cB(L^2(G))_* \}$. Note that $C_0(G)\subset L^\infty(G)$ is strongly dense. 
	The comultiplication $\Delta_G$ restricts to a non-degenerate $*$-homomorphism $C_0(G)\to \cM(C_0(G)\otimes C_0(G))$. 
	We also have the Pontryagin dual $\wh{G}=(L^\infty(\wh{G}),\wh{\Delta}_G)$ of $G$ as a locally compact quantum group such that we can canonically identify $L^2(\wh{G})$ with $L^2(G)$ and the corresponding multiplicative unitary, denoted by $\wh{W}^G$, is given by $W^{G*}_{21}$. 
	
	A \emph{unitary representation} of a locally compact quantum group $G$ is a pair $\pi=(\cH_\pi,U_\pi)$ of a Hilbert space $\cH_\pi$ and a unitary $U_\pi\in\cU\cM(\cK(\cH_\pi)\otimes C_0(G))$ such that $(\id\otimes\Delta_G)(U_\pi)=(U_{\pi})_{12}(U_\pi)_{13}$. 
	For two unitary representations $\pi,\varpi$ of $G$, we define their tensor product by $\pi\otimes\varpi:=(\cH_\pi\otimes\cH_\varpi, (U_{\pi})_{13}(U_{\varpi})_{23})$. 
	There is the universal version $C^u_0(G)$ of the C*-algebra $C_0(G)$, whose $*$-representations correspond to unitary representations of $\wh{G}$. 
	Also, $\Delta_G$ lifts to an injective non-degenerate coassociative $*$-homomorphism $\Delta^u_G\colon C^u_0(G)\to\cM(C^u_0(G)\otimes C^u_0(G))$. 
	We have a canonical surjective $*$-homomorphism $C^u_0(G)\to C_0(G)$ compatible with comultiplications. 
	If this is $*$-isomorphic, $G$ is called \emph{coamenable}. 
	See \cite{Kustermans:universal,Soltan-Woronowicz:mltu2}. 
	
	A \emph{compact quantum group} is a locally compact quantum group $G$ with unital $C_0(G)$. 
	In this case we write $C(G):=C_0(G)$ and $C^u(G):=C^u_0(G)$. 
	We can take $h_l$ and $h_r$ in \cref{def:LCQG} as a common state, which is called the \emph{Haar state}. 
	Finite dimensional unitary representations of $G$ form a rigid C*-tensor category denoted by $\Rep(G)$, and we write $\Irr G$ for the set of isomorphism classes of irreducible objects in $\Rep(G)$. 
	The fusion ring of $G$ is $R(G):=\bZ^{\oplus\Irr G}$ with the multiplication induced by the irreducible decompositions of tensor products. 
	We refer \cite{Neshveyev-Tuset:book} for more basic facts on compact quantum groups and their representation categories. 
	The Pontryagin dual $\wh{G}$ of a compact quantum group $G$ is called a \emph{discrete quantum group}. In this case, we write $c_0(\wh{G}):=C_0(\wh{G})$ and $\ell^\infty(\wh{G}):=L^\infty(\wh{G})$. They satisfy $c_0(\wh{G})\cong \bigoplus\limits_{\pi\in\Irr G}\cB(\cH_\pi)$ and $\ell^\infty(\wh{G})\cong \prod\limits_{\pi\in\Irr G}\cB(\cH_\pi)$, and $\wh{G}$ is coamenable. 
	
	We also utilize the notion of homomorphisms of locally compact quantum groups, which we refer \cite{Meyer-Roy-Woronowicz:qghom}. (Note the difference in the conventions.)
	\begin{df}
		For locally compact quantum groups $G$ and $H$, a \emph{homomorphism} $\phi\colon H\to G$ is given by a bicharacter, which is defined as a unitary $W^\phi\in\cU\cM(C_0(H)\otimes C_0(\wh{G}))$ such that 
		\begin{align*}
			&
			W^G_{23}W^{\phi}_{12} = W^{\phi}_{12}W^{\phi}_{13}W^G_{23}, 
			\quad\text{and}\quad
			W^{\phi}_{23}W^H_{12} = W^{H}_{12}W^{\phi}_{13}W^{\phi}_{23}. 
		\end{align*}
	\end{df}
	A bicharacter $W^\phi$ bijectively corresponds to a non-degenerate $*$-homomorphism compatible with comultiplications, denoted by $\phi^*\colon C^u_0(G)\to \cM(C^u_0(H))$. 
	Compositions of homomorphisms are given by compositions of the corresponding $\phi^*$. 
	For a bicharacter $W^\phi$, the unitary $W^{\phi*}_{21}$ is also a well-defined bicharacter and we write $\wh{\phi}\colon \wh{G}\to \wh{H}$ for the corresponding homomorphism. 
	The multiplicative unitary satisfies $W^G\in \cM(C_0(G)\otimes C_0(\wh{G}))$ and is a bicharacter corresponding to the identity map on $C^u_0(G)$. We write $\id_G\colon G\to G$ for the homomorphism given by $W^G$. 
	We have a trivial bicharacter $1\in \cU\cM(\bC\otimes C_0(\wh{G}))$, which corresponds to the $*$-homomorphism denoted by $\epsilon_G\colon C^u_0(G)\to \bC$ called the \emph{counit} of $G$. The counit factors through $C_0(G)$ if and only if $G$ is coamenable. 
	Especially, every locally compact group is coamenable. 
	
	It is convenient to use generalized quantum doubles associated to homomorphisms, and we refer \cite{Baaj-Vaes:doublecrossed,Roy:double} for their operator algebraic treatments. 
	\begin{df}
		Let $\phi\colon H\to G$ be a homomorphism of locally compact quantum groups. 
		We define 
		\begin{align*}
			&
			\rD(\phi):=(L^\infty(H)\barotimes L^\infty(\wh{G}),(\id\otimes(\sigma\Ad W^\phi)\otimes\id)(\Delta_H\otimes\wh{\Delta}_G)) , 
		\end{align*}
		which is a well-defined locally compact quantum group called the \emph{generalized quantum double}. 
		Here $\sigma\colon L^\infty(G)\barotimes L^\infty(H)\cong L^\infty(H)\barotimes L^\infty(G)$ is the flipping map $x\otimes y\mapsto y\otimes x$. 
	\end{df}
	
	We note $C_0(\rD(\phi))=C_0(H)\otimes C_0(\wh{G})\subset L^\infty(\rD(\phi))$. 
	If $H$ and $\wh{G}$ are coamenable, then $\rD(\phi)$ is also coamenable by the counit $\epsilon_H\otimes\epsilon_{\wh{G}}$. 
	We put $\rD(G):=\rD(\id_G)$ and call it the quantum double of $G$. 
	
	Closed quantum subgroups are defined as special cases of homomorphisms. 
	We refer \cite{Daws-Kasprzak-Skalski-Soltan:closed} for the definitions and properties. 
	For a homomorphism $\phi\colon H\to G$ of locally compact quantum groups, we say 
	$\phi$ gives a \emph{closed quantum subgroup in the sense of Vaes} if $\wh{\phi}^*$ induces an injective normal $*$-homomorphism $L^\infty(\wh{G})\to L^\infty(\wh{H})$, and 
	$\phi$ gives a \emph{closed quantum subgroup in the sense of Woronowicz} if $\phi^*\colon C^u_0(G)\to C^u_0(H)$ is surjective. 
	The former condition implies the latter, and if $H$ is compact or $G$ is discrete the converse holds. 
	A closed quantum subgroup in the sense of Woronowicz of a discrete quantum group is discrete. 
	Homomorphisms $\phi\colon H\to G$, $\psi\colon G\to F$, and $\theta\colon K\to H$ of locally compact quantum groups induce a homomorphism $\rD(\psi\phi\theta)\to \rD(\phi)$. 
	It gives a closed quantum subgroup in the sense of Vaes if $\wh{\psi}$ and $\theta$ do. 
	Also, $H$ and $\wh{G}$ are closed quantum subgroups of $\rD(\phi)$ in the sense of Vaes. 
	Finally, if $\phi^*\colon C^u_0(G)\to C^u_0(H)$ is bijective, then $\phi$ gives the closed quantum subgroup in the sense of Vaes and so is the homomorphism $G\to H$ corresponding to $(\phi^*)^{-1}$. 
	
	In this paper, we always use the following notation and terminologies unless stated otherwise. 
	The symbol $K$ denotes a connected simply connected compact semisimple Lie group. 
	We take its maximal torus $T$ and fix $0<q<1$. 
	We identify $T$ as a closed quantum subgroup $\iota_q\colon T\to K_q$ in the sense of Vaes in a canonical way. 
	Note that $K_q$ is coamenable. 
	For a discrete quantum group $\Gamma$ and a locally compact quantum group $G$, we would like to say a homomorphism $\phi\colon \Gamma\to G$ giving a closed quantum subgroup in the sense of Woronowicz is a discrete quantum subgroup. 
	Fortunately, discrete quantum subgroups in quantum doubles of compact quantum groups are automatically closed quantum subgroups in the sense of Vaes by \cref{prop:DQGclVW}. 
	Since most of other closed quantum subgroups we encounter in this paper are in the sense of Vaes, we would like to shortly call them closed quantum subgroups. 
	
	\section{Discrete quantum subgroups}\label{sec:nolattice}
	
	We recall from \cite{Daws-Kasprzak-Skalski-Soltan:closed} the \emph{Fourier algebra} of a locally compact quantum group $G$, defined by 
	\begin{align*}
		&
		\cA_G:=\{ (\id\otimes\omega)(W^G) \,|\, \omega\in L^\infty(\wh{G})_* \}\subset C_0(G) . 
	\end{align*}
	This is isomorphic to $L^\infty(\wh{G})_*$ as a vector space (see \cite[Lemma~1.6]{Daws-Kasprzak-Skalski-Soltan:closed}) and we regard $\cA_G$ as a Banach space with the norm from $L^\infty(\wh{G})_*$. 
	For a compact quantum group $G$, we write $\cO(G)\subset C(G)$ for the dense Hopf $*$-subalgebra linearly spanned by all matrix coefficients of finite dimensional unitary representations of $G$, 
	and $c_c(\wh{G}):=\bigoplus\limits_{\pi\in\Irr G} \cB(\cH_\pi)$ as an algebraic direct sum. 
	It holds $C^u(G)$ is the universal C*-completion of $\cO(G)$. 
	
	\begin{lem}\label{lem:denseFourier}
		For a homomorphism $\phi\colon H\to G$ of compact quantum groups, 
		we have a norm-dense inclusion $\cO(H)\odot c_c(\wh{G})\subset \cA_{\rD(\phi)}$. 
	\end{lem}
	
	For a compact quantum group $G$ we write the dense subspace $L^2_f(G)\subset L^2(G)$ as the image of $\cO(G)$ via the map $\Lambda\colon L^\infty(G)\to L^2(G)$ from the GNS construction of the Haar state. 
	We define a norm-dense subspace $\cF(G)\subset \cB(L^2(G))_*$ by the linear span of the linear functionals of the form $\cB(L^2(G))\ni T\mapsto \langle\xi, T\eta\rangle\in \bC$ for $\xi,\eta\in L^2_f(G)$. 
	We have a symmetry $U^G\in\cU(L^2(G))$ defined by $U^G\Lambda(x):=\Lambda(R\sigma_{i/2}(x))$ for $x\in \cO(G)$, where $(\sigma_t)_{t\in\bR}$ is the modular automorphism group of the Haar state on $L^\infty(G)$ and $R$ is the unitary antipode of $G$. 
	This $U^G$ is the product of modular conjugations of the left Haar weights of $\wh{G}$ and $G$ and preserves $L^2_f(G)$. 
	Note that for a homomorphism $\phi^*\colon H\to G$ of compact quantum groups, 
	$W^{\phi}_{21}=\bigoplus\limits_{\pi\in \Irr G} (\Ad U^G\otimes\phi^*) (U_{\pi}^{\oplus\dim\cH_\pi})$ 
	via the irreducible decomposition of the right regular unitary representation on $L^2(G)$. 
	Thus $W^{\phi}$ and $W^{\phi*}$ restrict to mutually inverse linear endomorphisms on $L^2_f(H)\odot L^2_f(G)$. 
	
	\begin{proof}
		We note $W^{\rD(\phi)} = W^{H}_{13} U^H_3W^{\phi*}_{32}U^{H*}_3 W^{G*}_{42}W^{\phi*}_{32}$ by \cite[Theorem~5.3, Proposition~8.1]{Baaj-Vaes:doublecrossed}. It holds 
		\begin{align}\label{eq:densesubqd}
			&
			\begin{aligned}
				&
				\{ (\id^{\otimes2}\otimes\omega)(W^{\rD(\phi)}) \,|\,\omega\in \cF(H\times G) \}
				\\=&
				\{ (\id^{\otimes2}\otimes\omega) (W^{H}_{13} U^H_3W^{\phi*}_{32}U^{H*}_3 W^{G*}_{42}W^{\phi*}_{32}) \,|\,\omega\in \cF(H\times G) \}
				\\=&
				{\mathrm{span}}
				\{ (\id^{\otimes2}\otimes\omega) (W^{H}_{13} U^H_3W^{\phi*}_{32}U^{H*}_3 x_2W^{\phi*}_{32}) \,|\,\omega\in \cF(H), x\in c_c(\wh{G}) \}
				\\=&
				\cO(H)\otimes c_c(\wh{G}), 
			\end{aligned}
		\end{align}
		where we have used the fact that the linear span of $c_c(\wh{G})_1 W^{\phi*}_{21} L^2_f(H)_2$ is $c_c(\wh{G})\odot L^2_f(H)$, which follows from the invertibility of the transposition of $U_\pi$ for each $\pi\in\Irr G$. 
	\end{proof}
	
	We need the following auxiliary result, which the author could not find appropriate references. 
	For a locally compact quantum group $G$, we write $\cW^G\in \cU\cM(C^u_0(G)\otimes C_0(\wh{G}))$ for the lift of $W^{G}$, 
	given in such a way that $\cW^{G*}$ is a maximal unitary representation of $\wh{G}$ (see \cite[Section~3]{Kustermans:universal}, \cite[Section~5]{Woronowicz:mltu}). 
	We have the following injective contractive algebra homomorphism with norm-dense range, 
	\begin{align}\label{eq:embedcAGuniv}
		&
		\cA_{\rD(\phi)}\cong L^\infty(\wh{\rD}(\phi))_* \ni \omega\mapsto (\id\otimes\omega)(\cW^{\rD(\phi)})\in C^u_0(\rD(\phi)). 
	\end{align}
	
	\begin{prop}\label{prop:univQD}
		For a homomorphism $\phi\colon H\to G$ of compact quantum groups, 
		we have a $*$-isomorphism $f\colon C^u_0(\rD(\phi))\cong C^u(H)\otimes c_0(\wh{G})$ such that $(f\otimes f)\Delta^u_{\rD(\phi)} = (\sigma\Ad \cW^\phi)_{23}(\Delta^u_H\otimes\wh{\Delta}_{G})f$, 
		where $\cW^\phi:=(\phi^*\otimes\id)(\cW^G)$. 
	\end{prop}
	
	\begin{proof}
		We write $\cF(\wh{\rD}(\phi))$ for the image of $\cF(H\times G)$ in $L^\infty(\wh{\rD}(\phi))_*$ via the restriction. 
		Since the dense subalgebra $\cO(H)\odot c_c(\wh{H})\subset C_0(\rD(\phi))$ is $*$-closed, 
		for any $\omega\in \cF(\wh{\rD}(\phi))$, there is a unique $\wt{\omega}\in \cF(\wh{\rD}(\phi))$ such that $(\id\otimes\wt{\omega})(W^{\rD(\phi)*}) = (\id\otimes\omega)(W^{\rD(\phi)*})^*$ by \cref{eq:densesubqd}. 
		By \cite[Proposition~8.3]{Kustermans-Vaes:lcqg} and \cite[Theorem~1.6.4]{Woronowicz:mltu}, using the antipode $\wh{S}$ of $\wh{\rD}(\phi)$ we see the partial map $\omega^*\wh{S}$ densely defined on $C_0(\wh{\rD}(\phi))$ extends normally to $\wt{\omega}$ on $L^\infty(\wh{\rD}(\phi))$, where $\omega^*(x):=\overline{\omega(x^*)}$ for $x\in L^\infty(\wh{\rD}(\phi))$. 
		Indeed, for any $\nu\in L^\infty(\rD(\phi))_*$, 
		\begin{align*}
			&
			(\nu\otimes \wt{\omega})(W^{\rD(\phi)*}) 
			= (\nu\otimes\id) ( (\id\otimes\omega)(W^{\rD(\phi)*})^* ) 
			\\=& (\nu\otimes \omega^*)(W^{\rD(\phi)}) 
			=(\nu\otimes \omega^*\wh{S})(W^{\rD(\phi)*}) . 
		\end{align*}
		Therefore when we fix a faithful $*$-representation of $C^u_0(\rD(\phi))$ to some Hilbert space $\cH$, similarly we see for any $\nu\in \cB(\cH)_*$, 
		\begin{align*}
			&
			(\nu\otimes\wt{\omega})(\cW^{\rD(\phi)*}) 
			= (\nu\otimes\omega^*\wh{S})(\cW^{\rD(\phi)*}) 
			\\=& (\nu\otimes\omega^*)(\cW^{\rD(\phi)}) 
			= (\nu\otimes\id) ((\id\otimes\omega)(\cW^{\rD(\phi)*})^*) . 
		\end{align*}
		
		Therefore \cref{eq:embedcAGuniv} is $*$-linear on $\cO(H)\odot c_c(\wh{G})$, 
		which induces a surjective $*$-homomorphism $g\colon C^u(H)\otimes c_0(\wh{G})\to C^u_0(\rD(\phi))$ from the universal C*-completion. 
		Note that when we put 
		\begin{align*}
			&
			\cX:=\cW^H_{13} U^H_3W^{\phi*}_{32}U^{H*}_3 W^{G*}_{42}W^{\phi*}_{32} 
			\in \cU\cM(C^u(H)\otimes c_0(\wh{G})\otimes C_0(\wh{\rD}(\phi)) ) , 
		\end{align*}
		a similar calculation to \cref{eq:densesubqd} shows 
		\begin{align*}
			&
			\cO(H)\odot c_c(\wh{G})=\{ (\id\otimes\omega)(\cX) \,|\, \omega\in\cF(H\times G) \} 
		\end{align*}
		as a $*$-subalgebra of $C^u(H)\otimes c_0(\wh{G})$. 
		By construction, $(g\otimes\omega)(\cX) = (\id\otimes\omega)(\cW^{\rD(\phi)})$ 
		for $\omega\in\cF(\wh{\rD}(\phi))$. 
		
		On the other hand, when we take a faithful $*$-representation of $C^u(H)$ to some Hilbert space $\cH_1$, the unitary $\cX^*\in \cU\cM(\cK(\cH_1\otimes L^2(G))\otimes C_0(\wh{\rD}(\phi)) )$ is a unitary representation of $\wh{\rD}(\phi)$, and we have a surjective $*$-homomorphism $f\colon C^u_0(\rD(\phi))\to C^u(H)\otimes c_0(\wh{G})$ such that 
		$(f\otimes\id)(\cW^{\rD(\phi)})=\cX$ by the universality of $C^u_0(\rD(\phi))$. 
		Thus $fg=\id$ on $\cO(H)\odot c_c(\wh{G})$, and $g$ is also injective. 
		
		Finally, by checking 
		\begin{align*}
			&
			(\sigma\Ad \cW^{\phi})_{23} (\Delta^u_H\otimes\wh{\Delta}_G\otimes\id^{\otimes2}) (\cX) = \cX_{1256}\cX_{3456} , 
		\end{align*}
		the compatibility of $f$ and the comultiplications follows. 
	\end{proof}
	
	\begin{prop}\label{prop:DQGclVW}
		Let $\phi\colon H\to G$ be a homomorphism of compact quantum groups, $\Gamma$ be a discrete quantum group, and $\psi\colon\Gamma\to \rD(\phi)$ give a closed quantum subgroup in the sense of Woronowicz. 
		Then $\psi$ gives a closed quantum subgroup in the sense of Vaes. 
	\end{prop}
	
	The proof can be done with a similar strategy as in \cite[Theorem~6.2]{Daws-Kasprzak-Skalski-Soltan:closed}. 
	Before the proof, we note that for a locally compact quantum group $G$, the canonical $*$-homomorphism with strongly dense range $C^u_0(\wh{G})\to C_0(\wh{G})\subset L^\infty(\wh{G})$ induces an isometry $\cA_G\cong L^\infty(\wh{G})_*\to C^u_0(\wh{G})^*$. 
	Also, using the lift $\fW^{G}\in\cU\cM(C^u_0(G)\otimes C^u_0(\wh{G}))$ of $W^{G}$ (see \cite[Section~3]{Kustermans:universal}, \cite[Section~4]{Meyer-Roy-Woronowicz:qghom}), we have an embedding 
	\begin{align}\label{eq:embedcBG}
		C^u_0(\wh{G})^*\ni \omega\mapsto (\id\otimes\omega)(\fW^{G})\in C^u_0(G) , 
	\end{align}
	similar to $\cA_G$, whose image regarded as the Fourier--Stieltjes algebra of $G$ (see the remark after \cite[Proposition~1.7]{Daws-Kasprzak-Skalski-Soltan:closed}). 
	
	\begin{proof}
		We identify $C^u(H)\otimes c_0(\wh{G})\cong C^u_0(\rD(\phi))$ by \cref{prop:univQD}. 
		Since $p_\pi:=\psi^*(1\otimes 1_{\cH_\pi})$ is a projection for each $\pi\in\Irr G$, it lies in some finite dimensional direct summand of $c_0(\Gamma)$. 
		Also, $p_\pi\in \cZ(c_0(\Gamma))$ by the surjectivity of $\psi^*$. 
		Moreover, 
		\begin{align*}
			&
			\psi^*(C^u(H)\otimes\cB(\cH_\pi)) = p_\pi\psi^*(C^u(H)\otimes c_0(\wh{G}))p_\pi = p_\pi c_0(\Gamma)p_\pi = p_\pi c_c(\Gamma)p_\pi , 
		\end{align*}
		which coincides with $\psi^*(\cO(H)\otimes\cB(\cH_\pi))$ because a dense subspace of a finite dimensional Banach space is the whole. 
		Since $\bigoplus\limits_{\pi\in\Irr G}p_\pi$ gives the unit in $\cM(c_0(\Gamma))$, it follows 
		$\psi^*(\cO(H)\odot c_c(\wh{G}))= c_c(\Gamma)$. 
		
		If we consider the map $C^u_0(\wh{\rD}(\phi))^*\to C^u_0(\wh{\Gamma})^*$ defined by the linear dual of $\wh{\psi}^*$, 
		we can check it coincides with the restriction of $\psi^*\colon C^u_0(\rD(\phi))\to c_0(\Gamma)$ via the embeddings \cref{eq:embedcBG} by using $(\psi^*\otimes\id)(\fW^{\rD(\phi)})=(\id\otimes\wh{\psi}^*)(\fW^{\Gamma})$. 
		Thus by considering their restrictions to $L^\infty(\wh{\rD}(\phi))_*\cong \cA_{\rD(\phi)}$, we see $\psi^*(\cA_{\rD(\phi)})\subset\cA_{\Gamma}$ is norm-dense by \cref{lem:denseFourier}. 
		Hence $\psi$ gives a closed quantum subgroup in the sense of Vaes by \cite[Theorem~3.7]{Daws-Kasprzak-Skalski-Soltan:closed}. 
	\end{proof}
	
	A compact quantum group $G$ is called of \emph{Kac type} if its Haar state is tracial. 
	For a compact quantum group $G$, consider the quotient C*-algebra $A$ of $C^u(G)$ by the closed ideal generated by all $x\in C^u(G)$ such that $\tau(x^*x)=0$ for every tracial state $\tau$ on $C^u(G)$. Then $\Delta_G$ induces a well-defined comultiplication on $A$, which gives rise to a Kac type compact quantum group $H$. See \cite[Appendix~A]{Soltan:quantumBohr} for detail. 
	The quotient map induces a homomorphism $\iota\colon H\to G$ giving a closed quantum subgroup. 
	We say $H$ is the maximal Kac quantum subgroup of $G$. 
	
	\begin{lem}\label{lem:DQGmKac}
		Let $G$ be a compact quantum group and $\iota\colon H\to G$ be its maximal Kac quantum subgroup. 
		Then a discrete quantum subgroup of $\rD(G)$ is a discrete quantum subgroup of $\rD(\iota)$. 
	\end{lem}
	
	\begin{proof}
		Let $\phi\colon \Gamma\to \rD(G)$ give a discrete quantum subgroup. 
		By the first paragraph of the proof of \cref{prop:DQGclVW}, $\phi^*(C^u(G)\otimes \cB(\cH_\pi))$ is finite dimensional for each $\pi\in\Irr G$. 
		By definition, a finite dimensional $*$-representation of $C^u(G)$ factors through $A$ above and thus through $C^u(H)$ (see the remark after \cite[Proposition~A.1]{Soltan:quantumBohr}). 
		Thus $\phi^*$ factors through $\iota^*\otimes \id\colon C^u(G)\otimes c_0(\wh{G})\to C^u(H)\otimes c_0(\wh{G})$. 
	\end{proof}
	
	\begin{rem}\label{rem:subDQGtorus}
		When $G$ is $K_q$, its maximal Kac quantum subgroup is $T$ by \cite[Lemma~4.10]{Tomatsu:coidealPoissonbdry}, 
		which uses the classification of irreducible $*$-representations of $C(K_q)$, \cite[Section~3.6]{Korogodski-Soibelman1998:bookptI}. 
		A consequence of the classification result is that every finite dimensional irreducible $*$-representation of $C(K_q)$ is given by an evaluation at some point in $T$. 
		Combined with the finite dimensionality of $\phi^*(C(K_q)\otimes \cB(\cH_\pi))$, this fact gives a more direct proof of \cref{lem:DQGmKac} in the case of $\rD(K_q)$. 
		It follows a discrete quantum subgroup of $\rD(K_q)$ is that of $\rD(\iota_q)$. 
	\end{rem}
	
	Let $\phi\colon H\to G$ be a homomorphism of locally compact quantum groups. 
	Then we have an injective normal $*$-homomorphism $\Delta_\phi\colon L^\infty(G)\to L^\infty(G)\barotimes L^\infty(H)$ such that $(\Delta_\phi\otimes\id) (W^G)=W^{G}_{13}W^{\phi}_{23}$. 
	It holds $(\id\otimes\Delta_H)\Delta_\phi=(\Delta_\phi\otimes\id)\Delta_\phi$. 
	We refer \cite{Meyer-Roy-Woronowicz:qghom} for further properties. 
	We write $L^\infty(G)^H:=\{ x\in L^\infty(G) \,|\, \Delta_\phi(x)=x\otimes 1 \}$. 
	Then $\Delta_G$ restricts to a well-defined map $L^\infty(G)^H\to L^\infty(G)\barotimes L^\infty(G)^H$. 
	A normal state $\theta$ on $L^\infty(G)^H$ is called $G$-invariant if $(\omega\otimes \theta)\Delta_G = \theta$ for any normal state $\omega$ on $L^\infty(G)$. 
	In \cite[Definition~5.5]{Brannan-Chirvasitu-Viselter2020:actionslattices}, a lattice of a locally compact quantum group $G$ is defined to be a discrete quantum subgroup $\Gamma$ admitting a $G$-invariant normal state on $L^\infty(G)^\Gamma$. 
	
	\begin{thm}\label{thm:nolattice}
		For a compact quantum group $G$, there is a lattice of $\rD(G)$ in the sense of \cite{Brannan-Chirvasitu-Viselter2020:actionslattices} if and only if $G$ is of Kac type. 
	\end{thm}
	
	For a homomorphism $\phi\colon H\to G$ giving a closed quantum subgroup, it holds $L^\infty(\rD(G))^{\rD(\phi)}=L^\infty(G)^H\otimes 1_{L^\infty(\wh{G})}$. 
	Also, we can check that a normal state $\theta$ on $L^\infty(G)^H$ is $\rD(G)$-invariant if and only if it is $G$-invariant and $\wh{G}$-invariant in the sense of $(\id\otimes \theta)\Ad W^G_{21}(1\otimes x)=\theta(x)1$ for any $x\in L^\infty(G)^H$. 
	
	\begin{proof}
		When $G$ is of Kac type, $\wh{G}$ is a lattice by \cite[Theorem~5.6]{Brannan-Chirvasitu-Viselter2020:actionslattices}. 
		Assume $G$ is not of Kac type and let $\iota\colon H\to G$ be a maximal Kac quantum subgroup. 
		Suppose $\Gamma$ is a lattice of $\rD(G)$ in this sense, with a $\rD(G)$-invariant normal state $\theta$ on $L^\infty(\rD(G))^\Gamma$. 
		Then $\theta$ restricts to a $G$-invariant state on 
		$L^\infty(G)^H=L^\infty(\rD(G))^{\rD(\iota)}$ by \cref{lem:DQGmKac}. 
		For the Haar state $h$ of $G$, it holds $\theta=(h\otimes \theta) \Delta_{G}=h$ on $L^\infty(G)^H$. 
		Since the restriction of $h$ on $L^\infty(G)^H$ is not $\wh{G}$-invariant by the following \cref{lem:invstateab}, we get a contradiction. 
	\end{proof}
	
	\begin{lem}\label{lem:invstateab}
		Let $G$ and $H$ be compact quantum groups with $H$ being of Kac type and $\iota\colon H\to G$ give a closed quantum subgroup. 
		The restriction of the Haar state $h$ of $G$ to $L^\infty(G)^H$ is $\wh{G}$-invariant if and only if $G$ is of Kac type. 
	\end{lem}

	\begin{proof}
		If $H$ is trivial, this is due to \cite[Corollary~3.9]{Izumi2002:noncommPoisson} (see also \cite[Lemma~5.2]{Kalantar-Kasprzak-Skalski-Vergnioux2022:noncommutative}). Thus it suffices to show the Haar state $h$ on $L^\infty(G)$ is $\wh{G}$-invariant by assuming its restriction on $L^\infty(G)^H$ is $\wh{G}$-invariant. 
		It is enough to check $(\id\otimes h)\Ad U_\pi(1\otimes x)=1\otimes h(x)$ for all $\pi\in \Irr G$ and $x\in L^\infty(G)$ by the expression of $W^G_{21}$ before the proof of \cref{lem:denseFourier}. 
		
		For any $x\in \cO(G)$ and $\pi\in\Irr G$, we put the average $y:=(\id\otimes h_H)\Delta_\iota (x)\in L^\infty(G)^H$ over the Haar state $h_H$ of $H$ and 
		we take an orthonormal basis of $\cH_{\pi}$ to express $U_\pi=(u_{i,j})_{i,j}$. 
		Since $H$ is of Kac type, $(\iota^*(u_{j,i}))_{i,j}\in \cB(\cH_\pi)\otimes C(H)$ is also unitary. 
		Thus by using $\Delta_\iota = (\id\otimes\iota^*)\Delta_G$ on $\cO(G)$, we check for each $i,j$, 
		\begin{align*}
			&
			h\Bigl(\sum\limits_{k=1}^{\dim\cH_\pi} u_{i,k}xu^{*}_{j,k}\Bigr) 
			=
			(h\otimes h_H\iota^*)\Delta_G\Bigl(\sum\limits_{k=1}^{\dim\cH_\pi} u_{i,k}xu^{*}_{j,k}\Bigr) 
			\\=&
			(h\otimes h_H)\Bigl( \sum\limits_{k,l=1}^{\dim\cH_\pi} (u_{i,l}\otimes \iota^*(u_{l,k})) \Delta_{\iota}(x) (u^{*}_{j,l}\otimes \iota^*(u_{l,k}^*)) \Bigr) 
			\\=&
			(h\otimes h_H)\Bigl( \sum\limits_{k,l=1}^{\dim\cH_\pi} (u_{i,l}\otimes 1) \Delta_{\iota}(x) (u^{*}_{j,l}\otimes \iota^*(u_{l,k}^*u_{l,k})) \Bigr) 
			\\=&
			h\Bigl(\sum\limits_{l=1}^{\dim\cH_\pi} u_{i,l} y u^{*}_{j,l}\Bigr) 
			=
			\delta_{i,j} h(y)
			=
			\delta_{i,j} h(x). 
			\qedhere\end{align*}
	\end{proof}
	
	\section{Classification}\label{sec:classify}
	
	From now, we assume $\Rep(G)$ for a compact quantum group $G$ is \emph{skeletal} in the sense that each isomorphism class of objects in $\Rep(G)$ consists of only one object in $\Rep(G)$, by replacing $\Rep(G)$ with a monoidally equivalent one if necessary, for simplicity. 
	For a homomorphism $\phi\colon H\to G$ of locally compact quantum groups giving a closed quantum subgroup in the sense of Woronowicz, 
	we sometimes identify $C^u_0(H)$ as a quotient $C^u_0(G)/\Ker \phi^*$ up to a $*$-isomorphism preserving comultiplications and write $H\leq G$ to emphasize this identification. 
	The following lemma would be folklore, but we include the proof for convenience of the reader. 
	
	\begin{lem}\label{lem:clsubDQGfusion}
		For a compact quantum group $G$, we have one-to-one correspondences between the following. 
		\begin{enumerate}
			\item
			A closed quantum subgroup of $\wh{G}$ up to the equivalence relation that identifies 
			two closed quantum subgroup $\phi_i\colon \Gamma_i\to \wh{G}$ for $i=1,2$ 
			if there is a homomorphism $\psi\colon \Gamma_1\to \Gamma_2$ with invertible $\psi^*$ such that $\phi_2\psi=\phi_1$. 
			\item
			An idempotent complete rigid C*-tensor full subcategory $\cC$ of $\Rep G$. 
			\item
			A subset $S$ of $\Irr G$ that is closed under conjugations of representations such that $\bZ^{\oplus S}\subset R(G)$ is a unital subring. 
		\end{enumerate}
	\end{lem}
	
	\begin{proof}
		For a closed quantum subgroup $\phi\colon \Gamma\to \wh{G}$ with the injective $*$-homomorphism $L^\infty(\wh{\Gamma})\to L^\infty(G)$ induced by $\wh{\phi}^*$, 
		the well-defined tensor functor $\Rep\wh{\Gamma}\to \Rep G$ sending each $\pi\in\Rep \wh{\Gamma}$ to $(\cH_\pi, (\id\otimes\wh{\phi}^*)(U_\pi))$ preserves irreducibility of objects, and the image of $\Irr \wh{\Gamma}\to \Irr G$ satisfies the conditions of (3). 
		Conversely, let $S\subset\Irr G$ be a subset as in (3) and take the von Neumann subalgebra $M\subset L^\infty(G)$ generated by $\{ (\omega\otimes\id)(U_\pi) \,|\, \pi\in S, \omega\in\cB(\cH_\pi)^* \}$. Then $H:=(M,\Delta_G|_M)$ is a locally compact quantum group with the restriction of the Haar state of $G$ and the inclusion $M=L^\infty(H)\subset L^\infty(G)$ gives a closed quantum subgroup $\wh{H}$ of $\wh{G}$. 
		Now we can check these constructions between (1) and (3) are mutually inverse. 
		It is easy to see the bijective correspondence from (2) to (3) can be given by sending $\cC$ to $\Irr\cC\subset \Irr G$. 
	\end{proof}
	
	The following proposition allows us to untwist the quantum group structure of certain semidirect products in the classification. 
	Let $G$ be a compact quantum group, $H$ be a locally compact group, and 
	$\gamma\colon H\to \Aut(c_0(\wh{G}))$ be a continuous homomorphism preserving the comultiplication of $\wh{G}$. 
	We take the $*$-automorphism $\mrm\in\Aut(L^\infty(H)\barotimes \ell^\infty(\wh{G}))$ satisfying 
	\begin{align*}
		&
		\mrm(f\otimes x)\colon t\mapsto f(t)\otimes \gamma_t(x)
	\end{align*}
	for $f\in L^\infty(H)$, $x\in \ell^\infty(\wh{G})$, $t\in H$ and 
	\begin{align*}
		&
		\rD(\gamma):=(L^\infty(H)\barotimes \ell^\infty(G), (\sigma\mrm)_{23}(\Delta_{H}\otimes\wh{\Delta}_{G}))
	\end{align*}
	is a well-defined locally compact quantum group with $C_0(\rD(\gamma))=C_0(H)\otimes c_0(\wh{G})$ by \cite[Proposition~9.2]{Baaj-Vaes:doublecrossed}. 
	It is coamenable with the counit $\epsilon_H\otimes \epsilon_{\wh{G}}$. 
	If $H$ is discrete, then so is $\rD(\gamma)$ by the expression from \cite[Proposition~9.1]{Baaj-Vaes:doublecrossed}. 
	
	\begin{prop}\label{prop:untwistconn}
		Let $G$ be a compact quantum group, $H$ be a connected compact group, 
		and $\gamma\colon H\to \Aut(c_0(\wh{G}))$ be a continuous homomorphism preserving the comultiplication. 
		We put $\fH$ as the group $H$ with the discrete topology and consider $\delta\colon \fH\to \Aut(c_0(\wh{G}))$ given by $\gamma$. 
		Then any discrete quantum subgroup $\Gamma\leq\rD(\gamma)$ is given by 
		\begin{align}\label{eq:prop:untwistconn}
			&
			c_0(\Gamma)=\bigoplus\limits_{\pi\in \Irr G} C(S_\pi)\otimes \cB(\cH_\pi), 
		\end{align}
		as a quotient of $C_0(\rD(\gamma))$ for some $S\in \cS_f$, 
		where $\cS_f$ is a family of $S\subset\fH\times\Irr G$ 
		that is closed under conjugations and gives a subring $\bZ^{\oplus S}\subset \bZ^{\oplus\fH}\otimes R(G)$ 
		such that $S_\pi:=\{ t\in\fH \,|\, (t,\pi)\in S \}$ is finite for all $\pi\in\Irr G$. 
		Conversely, any $S\in \cS_f$ gives a well-defined discrete quantum subgroup of $\rD(\gamma)$ in the way above. 
	\end{prop}
	
	In the proof, we write 
	$\pr_\pi\colon \ell^\infty(\wh{G})\to \cB(\cH_\pi)$ for the projection onto the component corresponding to each $\pi\in\Irr G$ and $\ev_t\colon C(H)\to \bC$ for the evaluation at each $t\in H$. 
	Then a finite dimensional $*$-representation of $c_0(\wh{G})$ is a direct sum of $\pr_\pi$ for finitely many $\pi\in\Irr G$ possibly with multiplicities, which form a C*-tensor category such that $\pr_\pi\otimes \pr_\varpi:=(\pr_\varpi\otimes \pr_\pi)\wh{\Delta}_{G}$ for $\pi,\varpi\in\Irr G$. 
	The resulting C*-tensor category is monoidally equivalent to $\Rep(G)$. 
	
	\begin{proof}
		An irreducible $*$-representation of $c_0(\fH)\otimes c_0(\wh{G})$ is of the form $\ev_t\otimes \pr_\pi$ for $t\in \Irr\wh{\fH}=\fH$ and $\pi\in \Irr G$, which gives a equivalence of C*-categories $\Rep \wh{\rD}(\delta)$ and $\Rep(\wh{\fH}\times G)$. 
		
		Since $\gamma\colon H\to \Aut(c_0(\wh{G}))$ induces a continuous action of $H$ on the set of minimal projections in $\cZ(c_0(\wh{G}))$, it must be trivial by the connectivity of $H$. 
		Thus for each $t\in H$, the $*$-automorphism $\gamma_t$ on $c_0(\wh{G})$ preserves each direct summand $\cB(\cH_\pi)$ for $\pi\in\Irr G$, and is of the form $\Ad \bigoplus\limits_{\pi\in\Irr G}w_{t,\pi}$ for some $w_{t,\pi}\in\cU(\cH_\pi)$. 
		For $s,t\in\Irr\wh{\fH}$ and $\pi,\varpi\in \Irr G$, we see the unitary 
		$w_{t,\pi}^*\otimes 1_{\cH_{\varpi}}$ 
		gives the unitary equivalence of finite dimensional $*$-representations of $c_0(\fH)\otimes c_0(\wh{G})$ on $\cH_\pi\otimes\cH_\varpi$, 
		\begin{align*}
			&
			\bigl((\ev_s\otimes \pr_\pi)\otimes (\ev_{t}\otimes \pr_{\varpi})\bigr)\Delta_{\rD(\delta)} 
			\\\cong &
			\bigl((\ev_{s}\otimes \ev_{t})\otimes (\pr_\pi\otimes \pr_{\varpi})\bigr) (\Delta_{\fH}\otimes \wh{\Delta}_G) .
		\end{align*}
		Therefore we see the fusion rules of $\wh{\rD}(\gamma)$ and $\wh{\fH}\times G$ are the same. 
		
		By (3) of \cref{lem:clsubDQGfusion}, we see the quotient C*-algebra of $c_0(\fH)\otimes c_0(\wh{G})$ gives a discrete quantum subgroup of $\rD(\delta)$ if and only if it gives a discrete quantum subgroup of $\fH\times \wh{G}$. 
		Any $S\in\cS_f$ gives a discrete quantum subgroup $\Gamma$ of $\rD(\delta)$ with \cref{eq:prop:untwistconn}. 
		Since the image of $\ell^\infty(\fH)\otimes \cB(\cH_\pi)\to \ell^\infty(\Gamma)$ is finite dimensional for every $\pi\in\Irr G$, 
		the image will not change if we compose the map after $C(H)\otimes \cB(\cH_\pi)\to \ell^\infty(\fH)\otimes \cB(\cH_\pi)$. 
		The resulting surjective $*$-homomorphism $C(H)\otimes c_0(\wh{G})\to c_0(\Gamma)$ preserves the comultiplications by construction, and thus $\Gamma$ is a discrete quantum subgroup of $\rD(\gamma)$. 
		
		Conversely, take any discrete quantum subgroup $\phi\colon \Gamma\to \rD(\gamma)$. 
		For each $\pi\in\Irr G$, the $*$-homomorphism $C(H)\otimes \cB(\cH_\pi)\to c_0(\Gamma)$ given by the restriction of $\phi^*$ 
		extends normally to $\ell^\infty(\fH)\otimes \cB(\cH_\pi)\to \ell^\infty(\Gamma)$ with a finite dimensional image, because $C(H)\otimes \cB(\cH_\pi)$ has a finite dimensional image in $c_0(\Gamma)$ and there is a finite subset $F\subset H$ such that $C_0(H\setminus F)\otimes\cB(\cH_\pi)\subset \Ker\phi^*$. 
		Thus $\phi^*$ extends to a surjective normal map $\psi^*\colon \ell^\infty(\rD(\delta)) \to \ell^\infty(\Gamma)$ compatible with comultiplications. 
		By construction, 
		\begin{align*}
			&
			\psi^*(c_0(\fH)\otimes \cB(\cH_\pi)) = \phi^*(C(H)\otimes \cB(\cH_\pi)) \subset c_0(\Gamma) 
		\end{align*}
		for each $\pi\in\Irr G$, and thus $\psi^*(c_0(\rD(\delta)))= c_0(\Gamma)$, which identifies $\Gamma$ as a discrete quantum subgroup of $\rD(\delta)$. 
		If we take $S:=\{ (t,\pi)\in\fH\times\Irr G \,|\, \psi^*(\delta_t\otimes1_{\cH_\pi})\neq 0 \}$, it holds $S\in\cS_f$, and $c_0(\Gamma)$ is of the form \cref{eq:Gammaform}. 
	\end{proof}
	
	If we let $\gamma_t:=\Ad \bigl((\ev_t\otimes \id)(W^{\iota_q})\bigr)$ for $t\in T$, then $\rD(\iota_q)\cong \rD(\gamma)$. 
	Since the fusion rules of $K_q$ and $K$ are the same, we get the following consequence. 
	
	\begin{cor}\label{cor:subDQGcptquot}
		Discrete quantum subgroups of $\rD(K_q)$ correspond to discrete quantum subgroups of $T\times \wh{K}$ bijectively. 
	\end{cor}
	
	To give an explicit description of discrete quantum subgroups, we decompose a $\Rep K$ into simple module subcategories. For a C*-tensor category $\cC$ and its idempotent complete C*-tensor full subcategory $\cD\subset \cC$, we say an idempotent complete full C*-subcategory $\cM\subset\cC$ is a $\cD$-module subcategory if the restriction of the monoidal bifunctor on $\cC$ gives a well-defined bifunctor $\cD\times \cM\to\cM$. 
	
	Let $G$ be a compact group. 
	For a character $\chi\in \Irr\cZ(G)=\wh{\cZ}(G)$, we take $\Rep_\chi G$ as the full C*-subcategory of $\Rep G$ generated by $\Irr_\chi G:=\{ \pi\in \Irr G \,|\, \chi\leq \pi|_{\cZ(G)} \}$. 
	Note that $\pi\in\Irr_\chi G$ satisfies $\pi|_{\cZ(G)}=\chi^{\oplus \dim\cH_\pi}$ by Schur's lemma. 
	Also, $\pi\otimes \pi'\subset\Rep_{\chi\chi'} G$ and $\overline{\pi}\subset \Rep_{\overline{\chi}} G$ for $\chi,\chi'\in\wh{\cZ}(G)$, $\pi\in \Rep_\chi G$, and $\pi'\in \Rep_{\chi'} G$. 
	Especially, $\Rep_{\chi}G$ is a $\Rep_{\1}G$-module subcategory of $\Rep G$. 
	We often identify $\Rep_{\1}G$ as the image of $\Rep(G/\cZ(G))$ via the restriction by $G\to G/\cZ(G)$. 
	Since $\varpi\leq \varpi\otimes\overline{\pi}\otimes\pi$ and $\varpi\otimes\overline{\pi}\in \Rep_{\1}G$ for $\pi,\varpi\in \Irr_\chi G$ and $\chi\in\wh{\cZ}(G)$, 
	we have the following. 
	
	\begin{lem}\label{lem:repsimplemod}
		For a compact group $G$, $\chi\in\wh{\cZ}(G)$, and $\pi\in\Irr_\chi G$, 
		the full subcategory $\Rep_\chi G\subset\Rep G$ is the smallest as an idempotent complete $\Rep (G/\cZ(G))$-module full subcategory containing $\pi$. 
	\end{lem}
	
	For $\chi\in\wh{\cZ}(K)$, we put $\Irr_\chi K_q$ for the image of $\Irr_\chi K$ via $\Irr K \cong \Irr K_q$. 
	
	\begin{thm}\label{thm:classifyss}
		For a connected simply connected compact semisimple Lie group $K$ and $0<q<1$, 
		we decompose $K=\prod\limits_{k=1}^{n} K_k$ into connected simply connected compact simple Lie groups $K_1,\cdots,K_n$ and put $K_J:=\prod\limits_{k\in J}K_k$ for $J\subset\{1,\cdots,n\}$. 
		Then any discrete quantum subgroup $\Gamma\leq\rD(K_q)$ is of the form 
		\begin{align*}
			&
			c_0(\Gamma)=\bigoplus\limits_{\chi\in F}C(S+f(\chi))\otimes \bigoplus\limits_{\pi\in \Irr_\chi K_{J,q}} \cB(\cH_\pi)
		\end{align*}
		for a subset $J\subset\{1,\cdots,n\}$, a finite subgroup $S$ of the maximal torus $T$ of $K$, and a group homomorphism $f\colon F\to T/S$ from a subgroup $F$ of $\wh{\cZ}(K_J)$. 
	\end{thm}
	
	Before the proof, note that $K_J/\cZ(K_J)=\prod\limits_{k\in J}K_k/\cZ(K_k)$ for $J\subset\{1,\cdots,n\}$. 
	We write $\fT$ for $T$ with the discrete topology. 
	Also, we often use the restrictions of representations via the inclusion $K_k\to K$ and the projection $K\to K_k$ for each $k$. 
	
	\begin{proof}
		First, we fix any discrete quantum subgroup $\Gamma\leq T\times \wh{K}$. We identify $\Irr\wh{\Gamma}$ with an element of $\cS_f$ as in \cref{prop:untwistconn} for $\fT\times\Irr K$ with trivial $\gamma$. 
		
		Take any $t\otimes \pi\in \Irr\wh{\Gamma}$, and decompose $\pi=\pi_1|_{K}\otimes\cdots\otimes\pi_n|_{K}$ with $\pi_k\in \Irr K_k$ for each $k$. For any $k\in\{1,\cdots,n\}$ such that $\pi_k$ is not one-dimensional, 
		$\overline{\pi}_k\otimes\pi_k\in \Rep_{\1}K_k$ has a non-one-dimensional irreducible summand, and 
		\begin{align*}
			&
			(\overline{\pi}_k\otimes\pi_k)|_K
			\leq 
			\overline{t\otimes \pi}\otimes t\otimes \pi 
			\in \Rep \wh{\Gamma}. 
		\end{align*}
		It follows $\{e\}\times \Irr (K_k/\cZ(K_k))\subset \Irr\wh{\Gamma}$ by the simplicity of $K_k/\cZ(K_k)$. 
		We put 
		\begin{align*}
			J:=&
			\{ k\in\{1,\cdots,n \} \,|\, \{e\}\times\Rep (K_k/\cZ(K_k))\subset \Irr\wh{\Gamma} \}
			\\
			=&
			\{ k\in\{1,\cdots,n\} \,|\, t\otimes\pi\in \Irr\wh{\Gamma}, \pi|_{K_k}\not\in\Rep(K_k^{\rm ab}) \} 
		\end{align*}
		using the abelianization $K_k^{\rm ab}$, and it holds 
		\begin{align}\label{eq:Gammasandwich}
			&
			\{e\}\times\Irr (K_J/\cZ(K_J))\subset \Irr\wh{\Gamma}\subset \fT\times \Irr K^{\rm ab}_{\{1,\cdots,n\}\setminus J}\times \Irr K_J. 
		\end{align}
		
		Here, note that $K_{\{1,\cdots,n\}\setminus J}^{\rm ab}$ is trivial. 
		For $t\in T$ and $\chi\in \wh{\cZ}(K_J)$, we put $I(\chi,t):= \{ \pi\in \Irr_\chi K_J \,|\, t\otimes\pi \in \Irr\wh{\Gamma} \}$. 
		The additive full subcategory of $\Rep K_J$ generated by $I(\chi,t)$ is a $\Rep (K_J/\cZ(K_J))$-module subcategory and thus 
		$I(\chi,t)$ equals either empty or $\Irr_{\chi} K_J$ by \cref{lem:repsimplemod}. 
		For $\chi\in \wh{\cZ}(K_J)$, we take (possibly empty) finite subsets 
		$S_\chi:=\{ t\in T \,|\, I(\chi,t)\neq\varnothing \}\subset T$, 
		which satisfies 
		\begin{align}\label{eq:Gammaform}
			&
			\Irr\wh{\Gamma} = \coprod\limits_{\chi\in \wh{\cZ}(K_J)}S_\chi \times \Irr_\chi K_J . 
		\end{align}
		Then $e\in S_{\1}$, $S_{\chi}^{-1} = S_{\overline{\chi}}$, and $S_\chi S_{\chi'}\subset S_{\chi\chi'}$ for $\chi,\chi'\in \wh{\cZ}(K)$. If $S_{\chi'}\neq\varnothing$, it follows 
		\begin{align*}
			&
			S_\chi S_{\chi'} 
			\supset S_{\chi\chi'} S_{\overline{\chi'}} S_{\chi'} 
			\supset S_{\chi\chi'} e
			= S_{\chi\chi'} . 
		\end{align*}
		
		Conversely, for any family $(S_\chi)_{\chi\in \wh{\cZ}(K)}$ of finite subsets in $T$ with $e\in S_{\1}$, $S_\chi^{-1}=S_{\overline{\chi}}$, and $S_{\chi}S_{\chi'} =S_{\chi\chi'}$ for $\chi,\chi\in \wh{\cZ}(K_J)$ with $S_{\chi'}\neq\varnothing$, 
		the subset of $\fT\times \Irr K$ of the form \cref{eq:Gammaform} is contained in $\cS_f$. 
		Now for each $J\subset\{1,\cdots,n\}$, such a family $(S_{\chi})_\chi$ corresponds bijectively to the triple $(S,F,f)$ as in the statement. 
		Indeed, 
		\begin{align*}
			&
			(S_\chi)_\chi 
			\mapsto 
			(S_{\1}, \{ \chi\in \wh{\cZ}(K_J) \,|\, S_\chi\neq \varnothing \}, [\chi\mapsto S_\chi/S_{\1}] )
		\end{align*}
		and $(S,F,f)\mapsto (S + f(\chi))_{\chi\in F}$ are mutually inverse bijections. 
		
		Now \cref{cor:subDQGcptquot} completes the proof. 
	\end{proof}
	
	\section{Remarks on variants}\label{sec:variant}
	
	\begin{rem}\label{rem:grtw&fo}
		So far we considered $K_q$ for $0<q<1$, but the same arguments hold for $SU_{-q}(2)$. 
		More generally, the argument of \cref{thm:classifyss} will work for 
		a compact quantum group with the same fusion rule as that of $K$ whose maximal Kac quantum subgroup is identified with $T$, 
		such as the one appearing in \cite[Theorem~4.1]{Neshveyev-Yamashita:suntype}. 
	\end{rem}
	
	\begin{cor}\label{cor:classifySLq2C}
		For $0\neq q\in (-1,1)$, a discrete quantum subgroup $\Gamma\leq\rD(SU_q(2))$ is either one of the following forms for $n\in \bZ_{\geq 1}$. 
		\begin{itemize}
			\item
			$c_0(\Gamma) = C(\bZ/n\bZ)\otimes c_0(\wh{H})$ for $H=\{ \1 \}, SO_q(3), SU_q(2)$. 
			\item
			$c_0(\Gamma) = 
			C(2\bZ/2n\bZ)\otimes c_0(\wh{SO}_q(3))\oplus C((2\bZ+1)/2n\bZ)\otimes I$, 
			where $I$ is the kernel of $c_0(\wh{SU}_q(2))\to c_0(\wh{SO}_q(3))$ corresponding to the discrete quantum subgroup $\wh{SO}_q(3)\leq \wh{SU}_q(2)$. 
		\end{itemize}
		Here we identified $\bZ/n\bZ$ as a subgroup of $\bR/\bZ\cong T$. 
	\end{cor}

	Next, we consider the case slightly beyond quantum doubles. 
	Let $Z\leq \cZ(K)$ be a subgroup, regarded as a closed quantum subgroup $\jmath\colon Z\to K_q$. 
	Note that $(\id\otimes\jmath^*)\Delta_{K_q} = \sigma(\jmath^*\otimes\id)\Delta_{K_q}$ and $W^\jmath\in \cU(C(Z)\otimes \cZ(c_0(\wh{K}_q)))$. 
	Then we can check $\Delta_{K_q}$ restricts to $L^\infty(K_q)^Z\to (L^\infty(K_q)^{Z})^{\barotimes 2}$, giving rise to a compact quantum group denoted by $K_q/Z$. 
	It holds $C(K_q/Z)=C(K_q)^Z$, the fixed point C*-algebra of the $Z$-action. 
	Then $\iota_q$ induces a homomorphism $\overline{\iota}_q\colon T/Z\to K_q/Z$ giving a closed quantum subgroup. 
	We can check $\Ad W^{K_q}$ and $\Ad W^{\iota_q}$ restricts to $\Aut(C(K_q/Z)\otimes c_0(\wh{K}_q))$ and $\Aut(C(T/Z)\otimes c_0(\wh{K}_q)))$, respectively, which extend to $*$-automorphisms of corresponding von Neumann algebras. 
	Thus by \cite[Proposition~9.2]{Baaj-Vaes:doublecrossed}, we have well-defined locally compact quantum groups $\rD(K_q)/Z$ and $\rD(\iota_q)/Z$ with 
	\begin{align*}
		&
		C_0(\rD(K_q)/Z)=C(K_q/Z)\otimes c_0(\wh{K}_q),
		\quad\text{and}\quad
		C_0(\rD(\iota_q)/Z)=C(T/Z)\otimes c_0(\wh{K}_q),
	\end{align*}
	respectively. 
	Their comultiplications and counits are given by restricting those of $\rD(K_q)$ and $\rD(\iota_q)$. 
	Especially, they are coamenable. 
	
	\begin{rem}
		In this situation, we have an imprimitivity $(C(K_q)\rtimes Z, C(K_q/Z))$-bimodule $\cE:=(\ell^2(Z)\otimes C(K_q))^Z$, since the action of $Z$ on the C*-algebra $C(K_q)$ is free. 
		For any $*$-representation of $C(K_q/Z)$ on a finite dimensional Hilbert space $\cH$, we get a unital $*$-representation of $C(K_q)\rtimes Z$ on $\wt{\cH}:=\cE\otimes_{C(K_q/Z)}\cH$, but $\cK(\wt{\cH})$ is unital since so are $\cK(\cE)$ and $\cK(\cH)$. 
		Then $\wt{\cH}$ is finite dimensional, and thus the $*$-representation of $C(K_q)\rtimes Z$ on $\wt{\cH}$ goes through $C(T)\rtimes Z$ by \cref{rem:subDQGtorus} and $Z$-equivariance of $\iota_q^*$. 
		Since 
		\begin{align*}
			&
			\cE^*\otimes_{C(K_q)\rtimes Z}C(T)\rtimes Z 
			\\\cong& 
			\bigl( (\ell^2(Z)^*\otimes C(K_q)) \otimes_{\cK(\ell^2(Z))\otimes C(K_q)} (\cK(\ell^2(Z))\otimes C(T)) \bigr)^Z
			\\\cong &
			(\ell^2(Z)^*\otimes C(T))^Z
		\end{align*}
		by the freeness of $Z$-actions, 
		it is an imprimitivity $(C(T/Z),C(T)\rtimes Z)$-bimodule. 
		This means the $C(K_q/Z)$-action on $\cH\cong \cE^*\otimes_{C(K_q)\rtimes Z}\wt{\cH}$ goes through $C(T/Z)$. 
		
		If we let $\gamma_t:=\Ad \bigl((\ev_t\otimes1) (W^{\iota_q})\bigr)$ for $t\in T$, this induces a well-defined continuous action $\overline{\gamma}\colon T/Z\to \Aut(c_0(\wh{K}_q))$ compatible with the comultiplication such that $\rD(\iota_q)/Z\cong \rD(\overline{\gamma})$. 
		It follows discrete quantum subgroups of $\rD(K_q)/Z$ are those of $\rD(\overline{\gamma})$. 
		We can still apply \cref{prop:untwistconn}, and the classification of discrete quantum subgroups of $\rD(K_q)/Z$ boils down to those of $T/Z \times \wh{K}_q$, 
		which correspond to discrete quantum subgroups of $T\times \wh{K}$ containing $Z\times\{\1_K\}$. 
	\end{rem}
	
	Similarly, discrete quantum subgroups of $\rD(K_q/Z)$ are those of $\rD(\overline{\iota}_q)$, which correspond to those of $T/Z\times \wh{K_q/Z}$. 
	We note the consequences for $\rD(SU_q(2))/\cZ(SU(2))$ and $\rD(SO_q(3))$, which are constructed as variants of the quantum Lorentz group in \cite[Section~6]{Podles-Woronowicz1990:quantum}. 
	
	\begin{cor}\label{cor:classifyLorentz}
		For $0<q<1$, a discrete quantum subgroup of $\rD(SU_q(2))/\cZ(SU(2))$ is one of those from the list in \cref{cor:classifySLq2C} for some even number $n\in\bZ_{\geq 1}$. 
		A discrete quantum subgroup of $\rD(SO_q(3))$ is given by $c_0(\bZ/n\bZ)\otimes c_0(\wh{H})$ for some even number $n\in \bZ_{\geq 1}$ and $H=\{ \1 \}, SO_q(3)$.  
	\end{cor}
	
	Now we consider the case of $U_q(n)$ for $0<q<1$ and $n\geq 2$. 
	We follow the treatment from \cite[Section~9.2]{Klimyk-Schmudgen:book}, except that 
	we write $u_{i,j}:=u^i_j$ for the generators of $\cO(U_q(n))$ and also for their images in $\cO(SU_q(n))$, where $i,j\in\{1,\cdots,n\}$. 
	To deduce several properties quickly, we use the following picture of $U_q(n)$. 
	By putting $u\in C(U(1))$ for the unitary given by $U(1)\subset \bC$, we have two well-defined $*$-homomorphisms 
	\begin{align*}
		f\colon C^u(U_q(n))\to& C(U(1))\otimes C(SU_q(n)),
		\\
		u_{i,j}\mapsto& u\otimes u_{i,j}
	\end{align*}
	\begin{align*}
		g\colon C(U(1))\otimes C(SU_q(n))\to & \cB(\bC^{\oplus n})\otimes C^u(U_q(n)) 
		\\
		u\otimes 1\mapsto & \left(\begin{array}{ccccc}
			0&1&0&\cdots&0\\
			0&0&1&\cdots&0\\
			\vdots&\vdots&\vdots&\ddots&\vdots\\
			0&0&0&\cdots&1\\
			\cD_q&0&0&\cdots&0
		\end{array}\right)
		\\
		u\otimes u_{i,j}\mapsto & 1\otimes u_{i,j}
	\end{align*}
	by the universality. 
	Here $\cD_q\in\cO(U_q(n))$ is the quantum determinant, which is a unitary lying in the center. 
	It holds $gf(x)=1\otimes x$ for $x\in C^u(U_q(n))$, and thus $f$ is injective. 
	Moreover, if we consider the homomorphism $\imath\colon \bZ/n\bZ\cong \cZ(SU(n))\to SU_q(n)$ given by the closed subgroup of the maximal torus in $SU(n)$ and the inclusion $\jmath\colon \bZ/n\bZ\cong \{ \zeta\in \bC\,|\, \zeta^n=1 \}\subset U(1)$, 
	we can see that the image of $f$ equals 
	\begin{align*}
		&
		\left\{ x\in C(U(1))\otimes C(SU_q(n)) \,\left|\, 
		\begin{array}{l}
			\quad(\id\otimes \jmath^*\otimes\id)(\Delta_{U(1)}\otimes\id)(x) \\
			=(\id\otimes\imath^*\otimes\id)(\id\otimes\Delta_{SU_q(n)})(x)
		\end{array} \right.\right\}, 
	\end{align*}
	by using $(\id\otimes\imath^*)\Delta_{SU_q(n)}=\sigma (\imath^*\otimes \id)\Delta_{SU_q(n)}$. 
	Clearly $f$ gives a homomorphism denoted by $\phi\colon U(1)\times SU_q(n)\to U_q(n)$. 
	
	Using this picture, the Haar state of $U_q(n)$ is given by $hf$ for the Haar state $h$ of $U(1)\times SU_q(n)$, and since it is faithful on $C^u(U_q(n))$, the coamenability of $K_q$ follows. 
	The injectivity of $f$ shows $\wh{\phi}$ gives a discrete quantum subgroup. 
	Then $\Rep U_q(n)$ is identified as the C*-tensor full subcategory of $\Rep (U(1)\times SU_q(n))$ consisting of the objects whose restrictions to $\Rep \bZ/n\bZ$ via $(\jmath\times \imath)\delta$ are finite direct sums of $\1_{\bZ/n\bZ}$, where $\delta\colon \bZ/n\bZ\to (\bZ/n\bZ)^{\times2}$ is the map sending each $k\in\bZ/n\bZ$ to $(-k,k)$. 
	Thus the fusion rules and the classical dimensions of $\Rep U_q(n)$ and $\Rep U(n)$ are the same, because so are for $\Rep SU_q(n)$ and $\Rep SU(n)$ with the isomorphism $R(SU_q(n))\cong R(SU(n))$ commuting with the restrictions to $R(\cZ(SU(n)))$. 
	
	We regard the maximal torus $T$ of $U(n)$ as a closed quantum subgroup $\iota_q\colon T\to U_q(n)$ by 
	\begin{align*}
		\iota_q^*(u_{i,j})&:= (\delta_{i,j} u \text{ at the $i$th component})
		\in C(U(1))^{\otimes n}\cong C(T) . 
	\end{align*}
	The injective $*$-homomorphism induced by $C(U(1))\ni u\mapsto \cD_q\in C(U_q(n))$ gives a closed quantum subgroup $\bZ\cong \wh{U}(1)\to \wh{U}_q(n)$. 
	For $\chi\in\wh{\cZ}(U(n))$, we put $\Irr_\chi U_q(n)$ for the image of $\Irr_\chi U(n)$ via $\Irr U(n) \cong \Irr U_q(n)$. 
	
	\begin{thm}\label{thm:classifyGLqC}
		A discrete quantum subgroup $\Gamma\leq\rD(U_q(n))$ is either a discrete subgroup of $T\times \wh{U}(1)$, or given by 
		\begin{align*}
			&
			c_0(\Gamma)=\bigoplus\limits_{\chi\in F}C(S+f(\chi))\otimes \bigoplus\limits_{\pi\in \Irr_\chi U_q(n)} \cB(\cH_\pi)
		\end{align*}
		for a finite subgroup $S\subset T$ and a group homomorphism $f\colon F\to T/S$ from a subgroup $F\subset \wh{\cZ}(U(n))\cong \bZ$. 
	\end{thm}
	
	\begin{proof}
		Since there are mutually inverse well-defined $*$-isomorphisms induced by 
		\begin{align*}
			C(U_q(n)) \rightleftarrows &\, C(U(1))\otimes C(SU_q(n)) ,\\
			u_{i,j} \mapsto & u^{\delta_{i,1}}\otimes u_{i,j}\\
			u_{i,j}\cD_q^{m-\delta_{i,1}} \mapsfrom&\, u^m\otimes u_{i,j}
		\end{align*} 
		for $i,j\in\{1,\cdots,n\}$ and $m\in\bZ$, 
		we see any finite dimensional $*$-representation of $C(U_q(n))$ factors through $C(T)$ and thus any discrete quantum subgroup of $\rD(U_q(n))$ is that of $\rD(\iota_q)$ as in \cref{rem:subDQGtorus}. 
		
		Now the arguments of \cref{cor:subDQGcptquot} and \cref{eq:Gammasandwich} still work if we let $K=K_{1}=U(n)$. 
		Note that $U(n)^{\rm ab}\cong U(1)$ via the determinant map. 
		When $J=\varnothing$, a discrete quantum subgroup of $\rD(U_q(n))$ is that of the generalized quantum double of $T\leq U_q(n)\to U(1)$, which is isomorphic to $T\times \wh{U}(1)$. 
		If $J=\{1\}$, the remaining proof of \cref{thm:classifyss} is also valid for this situation, and we obtain the conclusion. 
	\end{proof}
	
	\providecommand{\bysame}{\leavevmode\hbox to3em{\hrulefill}\thinspace}

\end{document}